\documentclass[reqno]{amsart}
\usepackage{amssymb,graphicx}
\usepackage{tikz}
\usetikzlibrary{shapes,arrows}
\usetikzlibrary{shapes.geometric,arrows,chains}
\usepackage{mathrsfs}
\usepackage{ifpdf}
\usepackage[hidelinks]{hyperref}
\hypersetup{colorlinks=true,linkcolor=blue}
\usepackage{setspace}
\usepackage[fontsize=11.5pt]{fontsize}
\usepackage{bm}

\usepackage{geometry}
\allowdisplaybreaks[4]
\theoremstyle{plain}

\newtheorem{thm}{Theorem}[section]
\newtheorem{lem}[thm]{Lemma}
\newtheorem{cor}[thm]{Corollary}

\newtheorem{rmk}[thm]{Remark}

\numberwithin{equation}{section}

\def\B{B(x_0,R)}

\def\B2T{B(x_0,2R)\times[0,T]}
\def\d{\Delta}
\def\e{\epsilon}
\def\g{\gamma}

\def\M{\mathbf{M^n}}

\def\P{\Phi}
\def\p{\partial_t}

\def\LI{\mathcal{L}}
\def\L{\mathscr{L}}

\def\t{\Delta}

\def\tx{t^{\star}}
\def\x{x^{\star}}

\makeatother
\makeatletter
\@namedef{subjclassname@2020}{\textup{2020} Mathematics Subject Classification}
\makeatother

\begin{document}
	
	\title[Differential Harnack inequality]
	{Differential Harnack inequalities for Fisher-KPP \\type equations on Riemannian manifolds}
	
	\author[Zhihao Lu]{Zhihao Lu}
	\address[Zhihao Lu]{School of Mathematics and Statistics, Jiangsu Normal University, Xuzhou 221116, P. R. China}
	\email{lzh139@mail.ustc.edu.cn}

	\begin{abstract}
		We obtain almost optimal differential Harnack inequalities for a class of nonlinear parabolic equations on Riemannian manifolds with Bakry-\'{E}mery Ricci curvature bounded below, which includes the classical Fisher-KPP equation and Newell-Whitehead equation. Compared to existing research, we do not impose any additional conditions on the positive solutions. As its   application, we derive some optimal Liouville properties.
	\end{abstract}
	
	
	
	\thanks{School of Mathematics and Statistics, Jiangsu Normal University, Xuzhou 221116, P. R. China}
	
	
	\maketitle
	\section{Introduction}
	The following Fisher-KPP equation ($a>0$)
	\begin{equation}\label{kpp}
		\p u=\d u+au(1-u)
	\end{equation}
	is an important model to describe the propagation of an evolutionarily advantageous gene in a population. It was independently researched in Fisher \cite{F} and Kolmogorov-Petrovsky-Piskunov  \cite{KPP}. For this reason,
	it is often referred to in the literature as the Fisher–KPP equation. This equation also has many applications in the branching Brownian motion process, in neuropsychology and in describing certain chemical
	reactions. Generally, the solution $u$ represents a density or  concentration and so we naturally assume that $u$ is positive.

	In this paper, we are considering the Fisher-KPP type equation given by
	\begin{equation}\label{KPP}
		\partial_t u=\Delta_V u+au+bu^p,
	\end{equation}
	which is defined  on a domain of Riemannian manifold. Here $\Delta_Vu:=\d u+Vu$ ($V$ is a smooth vector field on manifold) and $a\ge 0$, $b\le 0$ and $p>1$ are real numbers. When $V$ vanishes, $a=-b>0$ and $p=2$, it reduces to the Fisher-KPP equation; when $a>0$, $b<0$ and $p=3$, it is known as the Newell-Whitehead equation, which includes the parabolic Allen-Cahn equation ($a=-b=1$).  In this paper, using the tool of differential Harnack inequality, we will investigate the behavior of solutions and establish some sharp Liouville properties for equation \eqref{KPP}.
	
	Concretely, we will establish differential Harnack inequalities for equation \eqref{KPP}, which are inspired by Li-Yau's seminal work \cite{LY} on the heat equation. So, these inequalities can be applied to equation \eqref{KPP} on Riemannian manifolds with Bakry-\'{E}mery Ricci curvature bounded below.  Nevertheless, all of the results presented below are new for equation \eqref{KPP} even on Euclidean spaces.
	
	Before we state the main results, let us establish some conventions. Throughout this paper, we always assume that the mentioned Riemannian manifolds  are complete and the mentioned solutions are classical. We will use the notation $B(x,R)$ to denote an open geodesic ball with center $x$ and radius $R$ in the Riemannian manifold. The letter $V$ is always denoted as a smooth vector field in a manifold.  The positive constants that depend on parameters will be denoted by $C(\cdot)$. We define an ancient solution as a solution that satisfies \eqref{KPP} on $\M\times(-\infty,0)$ and an eternal solution as a solution that satisfies \eqref{KPP} on $\M\times\mathbb{R}$. As standard notations, we define
	\begin{equation}
		Ric_V=Ric-\frac{1}{2}\mathcal{L}_V g,
	\end{equation}
	and for any $m>n$ ($n$ is the dimension of the manifold),
	\begin{equation}
		Ric_V^m=Ric_V-\frac{1}{m-n}V^{\flat}\otimes V^{\flat}.
	\end{equation}
Here, $\LI_V g$ is the Lie derivative of the metric with respect to $V$ and  $V^{\flat}$ is the dual 1-form of $V$. We  call $Ric_V$ and $Ric_V^m$ as $\infty$-dimensional and $m$-dimensional Bakry-Émery Ricci curvatures respectively. It is obvious that when $V$ vanishes, $Ric_V$ and $Ric_V^m$ are both Ricci curvature.

	First, we establish a differential Harnack inequality for equation \eqref{KPP} on Riemannian manifolds with Bakry-\'{E}mery Ricci curvature bounded below.  The condition $a>0$ plays a key role in absorbing the  possible positive constant at the right hand of the inequality. 
	
	\begin{thm}\label{dl}
		Let $(\M,g)$ be an n-dimensional Riemannian manifold with $Ric^m_V\ge-Kg$, where $K\ge0$ and $m>n$. Let $u$ be a positive solution of \eqref{KPP} on $B(x_0,2R)\times[0,T]$ with $a>0$, $b\le 0$ and $p>1$, then there exist  $\g=\g(m,p,a,K)> 0$ and $\delta=\delta(m,a,K)<0$ such that
		\begin{equation}\label{13}
			\left(\g\frac{|\nabla u|^2}{u^2}-\frac{u_t}{u}+\delta\left(a+bu^{p-1}\right)\right)(x,t)\le C\left(\frac{1}{t}+\frac{1+\sqrt{K}R}{R^2}\right)
		\end{equation}
		on $B(x_0,R)\times[0,T]$\footnote{Here, at $t=0$, we obey the convention that $\frac{1}{0}=\infty$.}, where $C=C(m,p,a,K)$.
	\end{thm}

	Therefore, we obtain the optimal upper bound for the positive ancient or eternal solutions of \eqref{KPP} on Riemannian manifolds with Bakry-\'{E}mery Ricci curvature bounded below.
	\begin{cor}\label{c12}
		Let $(\M,g)$ be an n-dimensional Riemannian manifold with $Ric^m_V\ge-Kg$, where $K\ge 0$ and $m>n$. If $u$ is a positive ancient solution of  \eqref{KPP} with $a>0$, $b<0$ and $p>1$, then $\limsup\limits_{t\to-\infty}u(x,t)\le\left(-\frac{a}{b}\right)^{\frac{1}{p-1}}$. Moreover, if $u$ is a positive eternal solution of  \eqref{KPP}, then $u\le\left(-\frac{a}{b}\right)^{\frac{1}{p-1}}$.
	\end{cor}

Therefore, Corollary \ref{c12} can be applied to equation \eqref{KPP} on Riemannian manifolds with $\infty$-dimensional Bakry-\'{E}mery Ricci curvature bounded below satisfying that $L^{\infty}$ norm of $V$ is finite. Specially,  for equation \eqref{KPP} on Euclidean spaces, Corollary \ref{c12} is suitable for vector satisfying that $L^{\infty}$ norm of $V$ and $\nabla V$  are finite.

	\begin{rmk}\rm 
		To best of our knowledge, the optimal bound in Corollary \ref{c12} is new even for \eqref{KPP} on the Euclidean spaces. Therefore, any positive solutions of elliptic equation of \eqref{KPP} ($a>0$ and $b<0$) on complete Riemannian manifolds with Ricci curvature bounded below must have upper bound $\left(\frac{a}{b}\right)^{\frac{1}{p-1}}$, and so any positive solutions of elliptic equation of \eqref{KPP} on closed manifolds must be constant. 
	\end{rmk}

	If $a=0$ and $b<0$, under Bakry-\'{E}mery Ricci non-negative condition,  we have 
	\begin{thm}\label{dl0}
		Let $(\M,g)$ be an n-dimensional Riemannian manifold with $Ric^m_V\ge0$ and $m>n$. Let $u$ be a positive solution of \eqref{KPP} on $B(x_0,2R)\times[0,T]$ with $a=0$, $b\le0$ and $p>1$, then there exists  $\g=\g(m,p)\in(0,1)$ such that
		\begin{equation}\label{15}
			\left(\g\frac{|\nabla u|^2}{u^2}-\frac{u_t}{u}-bu^{p-1}\right)(x,t)\le C\left(\frac{1}{t}+\frac{1}{R^2}\right)
		\end{equation}
		on $B(x_0,R)\times[0,T]$, where $C=C(m,p)$.
	\end{thm}
	
	For the elliptic equation of \eqref{KPP} with $a=0$ and $b<0$, we have
	
	\begin{thm}\label{dl0e}
		Let $(\M,g)$ be an n-dimensional Riemannian manifold with $Ric^m_V\ge-Kg$, where $K\ge0$ and $m>n$. Let $u$ be a positive solution of elliptic equation of \eqref{KPP} on $B(x_0,2R)$ with $a=0$, $b\le0$ and $p>1$, then
		\begin{equation}\label{16}
			\sup\limits_{B(x_0,R)}\frac{|\nabla u|^2}{u^2}-bu^{p-1}\le C\left(\frac{1}{R^2}+K\right),
		\end{equation}
		where $C=C(m,p)$.
	\end{thm}
	Theorem \ref{dl0e} not only provides the logarithmic gradient estimate, but also the uniform upper bound for the positive solutions of elliptic equation of \eqref{KPP} with $a=0$ and $b<0$. For $b>0$ case, similar estimate only holds for subcritical index $p$, see \cite{LU4}.
	
	The following two corollaries are direct results of Theorem \ref{dl0}.
	\begin{cor}\label{cl0}
		Let $(\M,g)$ be an n-dimensional closed manifold with $Ric^m_V\ge 0$ and $m>n$. Then there does not exist any positive ancient solution of \eqref{KPP} with $a=0$, $b<0$ and $p>1$. 
	\end{cor}
	
	\begin{cor}\label{nc0}
		Let $(\M,g)$ be an n-dimensional Riemannian manifold with $Ric^m_V\ge 0$ and $m>n$. Then  any positive ancient solution of \eqref{KPP} with $a=0$, $b<0$ and $p>1$ must satisfy that for any $x\in\M$, $u(x,t)=O((-t)^{\frac{1}{1-p}})$ as $t\to-\infty$. 
	\end{cor}

	Now, we present a differential Harnack inequality with a parameter $\delta>0$, similar to inequality \eqref{13}. 	First, for any real numbers $m\ge1$, $p\ge1+\frac{2}{m}$ and $a\ge 0$, we define
	\begin{eqnarray}\label{t1}
		L(m,p,a)=\begin{cases}
			\frac{(p-1)a}{2}\qquad\qquad\quad \text{if}\qquad p\in\left[1+\frac{2}{m},1+\frac{4}{m}\right)\\
			\frac{2a}{m}\qquad\qquad\qquad\,\,\,\,\text{if}\qquad p\in\left[1+\frac{4}{m},\infty\right).
		\end{cases}
	\end{eqnarray}	
	Under some restrictions of index $p$ and the lower bounds of Bakry-\'{E}mery Ricci curvature, then we have the following  stronger differential Harnack inequalities.

	\begin{thm}\label{dg}
		Let $(\M,g)$ be an n-dimensional Riemannian manifold with $Ric^m_V\ge-Kg$, where $K\ge 0$ and $m>n$. Let $u$ be a positive solution of \eqref{KPP} on $B(x_0,2R)\times[0,T]$ with $a>0$ and $b<0$.\\
		{\rm (1)} If $p\in \left(1+\frac{2}{m},\infty\right)$ and  $K\in[0,L(m,p,a))$, there exist $\g=\g(m,p,a,K)\in(0,1)$ and $\delta=\delta(m,p,a,K)\in(0,1)$ such that
		\begin{eqnarray}\label{19}
			\left(\g\frac{|\nabla u|^2}{u^2}-\frac{u_t}{u}+\delta|a+bu^{p-1}|\right)(x,t)	\le C\left(\frac{1}{t}+\frac{1+\sqrt{K}R}{R^2}\right)
		\end{eqnarray}
		on $B(x_0,R)\times[0,T]$, where $C=C(m,p,a,K)$.\\
		{\rm (2)} If $p\in\left(\max\left(\frac{4}{m},1\right),\infty\right)$ and $K=0$, there exist $\g=\g(m,p,a)\in(0,1)$ and $\delta=\delta(m,p,a)\in(0,1)$ such that
		\begin{eqnarray}\label{110}
			\left(\g\frac{|\nabla u|^2}{u^2}-\frac{u_t}{u}+\delta|a+bu^{p-1}|\right)(x,t)	\le C\left(\frac{1}{t}+\frac{1}{R^2}\right)
		\end{eqnarray}
		on $B(x_0,R)\times[0,T]$, where $C=C(m,p,a)$.
	\end{thm}

	The following Liouville property is a direct corollary of Theorem \ref{dg}.
	\begin{cor}\label{pliou}
		Let $(\M,g)$ be an n-dimensional Riemannian manifold with $Ric^m_V\ge-Kg$, where $K\ge 0$ and $m>n$. If $u$ is a positive ancient solution of \eqref{KPP} with $a>0$ and $b<0$, and at least one of the following two hold:\\
		{\rm(a)} $p\in(1+\frac{2}{m},\infty)$ and $K\in[0,L(m,p,a))$;\\
		{\rm (b)} $p\in\left(\max\left(\frac{4}{m},1\right),\infty\right)$ and $K=0$,\\
		then for any $x\in\M$, either one of the following two cases must occur:\\
		{\rm (1)} $\lim\limits_{t\to-\infty}u(x,t)= \left(-\frac{a}{b}\right)^{\frac{1}{p-1}}$;\\
		{\rm (2)} $u(x,t)=O(e^{ct})$ as $t\to -\infty$ for some $c>0$.
	\end{cor}

	Furthermore, using Theorem \ref{dg}, we also readily obtain the following sharp asymptotic behavior for solutions on positive time.
	
	\begin{cor}\label{pliouf}
		Let $(\M,g)$ be an n-dimensional Riemannian manifold with $Ric^m_V\ge-Kg$, where $K\ge 0$ and $m>n$. If $u$ is a positive  solution of \eqref{KPP} on $\M\times[0,\infty)$ with $a>0$ and $b<0$, and at least one of the following two hold:\\
		{\rm(a)} $p\in(1+\frac{2}{m},\infty)$ and $K\in[0,L(n,p,a))$;\\
		{\rm (b)} $p\in\left(\max\left(\frac{4}{m},1\right),\infty\right)$ and $K=0$,\\
		then for any $x\in\M$, $\lim\limits_{t\to\infty}u(x,t)=\left(-\frac{a}{b}\right)^{\frac{1}{p-1}}$.
	\end{cor}
	
	If we consider the eternal solutions, we have more precise decay behavior.
	\begin{cor}\label{plioue}
		Let $(\M,g)$ be an n-dimensional Riemannian manifold with $Ric^m_V\ge-Kg$, where $K\ge 0$ and $m>n$. If $u$ is a positive eternal solution of \eqref{KPP} with $a>0$ and $b<0$, and at least one of the following two hold:\\
		{\rm(a)} $p\in(1+\frac{2}{m},\infty)$ and $K\in[0,L(m,p,a))$;\\
		{\rm (b)} $p\in\left(\max\left(\frac{4}{m},1\right),\infty\right)$ and $K=0$,\\
		then for any $x\in\M$,  $u(x,t)=\left(-\frac{a}{b}\right)^{\frac{1}{p-1}}-O(e^{-ct})$ as $t\to\infty$, where $c$ is a positive number.
	\end{cor}
	
	From  Corollary \ref{plioue}, we have the following Liouville theorem for the  elliptic equation of \eqref{KPP}.
	\begin{cor}\label{eliou}
		Let $(\M,g)$ be an n-dimensional Riemannian manifold with $Ric^m_V\ge-Kg$, where $K\ge 0$ and $m>n$. If $u$ is a positive solution of   elliptic equation of \eqref{KPP} on $\M$ with $a>0$, $b<0$,  $p\in(1+\frac{2}{m},\infty)$ and   $K\in[0,L(m,p,a))$, then $u\equiv\left(-\frac{a}{b}\right)^{\frac{1}{p-1}}$.
	\end{cor}

	Now, using  Theorem \ref{dg}, for any $p>1$, we can choose $m=m(n,p)>n$ such that $1+\frac{2}{m}>p$, and so  under Ricci non-negative condition (that is $V=0$), we have 
	\begin{thm}\label{xs}
		Let $(\M,g)$ be an n-dimensional Riemannian manifold with $Ric\ge0$. If $u$ is a positive solution of \eqref{KPP} on $B(x_0,2R)\times[0,T]$ with $a>0$ and $b<0$, there exist $\g=\g(n,p,a)\in(0,1)$ and $\delta=\delta(n,p,a)>0$ such that 
		\begin{eqnarray}
			\g\frac{|\nabla u|^2}{u^2}-\frac{\p u}{u}+\delta|a+bu^{p-1}|	\le  C\left(\frac{1}{t}+\frac{1}{R^2}\right)
		\end{eqnarray}
		on $B(x_0,R)\times [0,T]$, where $C=C(n,p,a)$.
	\end{thm}
	By Theorem \ref{xs}, under Ricci nonnegative condition, we readily derive  the following optimal asymptotic behavior as Corollary \ref{pliou}-Corollary \ref{plioue}.
	\begin{thm}\label{xl}
		Let $(\M,g)$ be an n-dimensional Riemannian manifold with $Ric\ge0$.\\
		(1) If $u$ is a positive ancient solution of \eqref{KPP} with $a>0,b<0$, $p>1$ and $V=0$, then for any $x\in\M$, $u(x,t)=O(e^{ct})$ as $t\to -\infty$ for some $c>0$ unless $u$ is constant.\\
		(2) If $u$ is a positive solution of \eqref{KPP} on $\M\times[0,\infty)$ with $a>0,b<0$, $p>1$ and $V=0$, then for any $x\in\M$, $u(x,t)\to \left(-\frac{a}{b}\right)^{\frac{1}{p-1}}$ as $t\to \infty$.\\
		(3) If $u$ is a positive eternal solution of \eqref{KPP} with $a>0,b<0$, $p>1$ and $V=0$, then for any $x\in\M$, $u(x,t)=\left(-\frac{a}{b}\right)^{\frac{1}{p-1}}-O(e^{-ct})$ as $t\to \infty$ for some $c>0$.
	\end{thm}

	\begin{rmk}
		\rm
		 For the case $\delta>0$, we introduce the artificial number $L(m,p,a)$ based on the restrictions of our technique. It is natural to inquire whether this lower bound restriction on the Ricci curvature is necessary. If it is indeed necessary, can one improve it to an optimal bound?
	\end{rmk}

	Actually, the restriction on the lower bound of Ricci curvature is not necessary if we consider the corresponding elliptic equation with $a>0$ and $b<0$. However, it is important to note that the constant $C$ will still depend on the lower bound $-K$ of the Ricci curvature.
	
	\begin{thm}\label{dge}
		Let $(\M,g)$ be an n-dimensional Riemannian manifold with $Ric_V^m\ge-Kg$, where $K\ge 0$ and $m>n$. If $u$ is a positive solution of elliptic equation of \eqref{KPP} on $B(x_0,2R)$ with $a>0$, $b<0$ and $p>1$, then 
		\begin{eqnarray}
			\sup\limits_{B(x_0,R)}\frac{|\nabla u|^2}{u^2}+a+bu^{p-1}\le C\left(\frac{1}{R^2}+K\right),
		\end{eqnarray}
		 where $C=C(m,p,a,K)$.
	\end{thm}

	Combining Theorem \ref{dl} and Theorem \ref{dge}, we have the following logarithmic gradient estimate and uniform bound estimate for elliptic equation of \eqref{KPP}. 
	
	\begin{thm}\label{e}
		Let $(\M,g)$ be an n-dimensional Riemannian manifold with $Ric_V^m\ge-Kg$, where $K\ge 0$ and $m>n$. If $u$ is a positive solution of elliptic equation of \eqref{KPP} on $B(x_0,2R)$ with $a>0$, $b<0$ and $p>1$, then 
		\begin{eqnarray}\label{113}
			\sup\limits_{B(x_0,R)}\frac{|\nabla u|^2}{u^2}+|a+bu^{p-1}|	\le C\left(\frac{1}{R^2}+K\right),
		\end{eqnarray}
		where $C=C(m,p,a,K)$.
	\end{thm}

	\begin{rmk}
		\rm (1) Here, we point out that Theorem \ref{dl} first provides a universal upper bound for any positive solutions for the elliptic equation of \eqref{KPP}, which is crucial in deriving the universal lower bound of solutions in Theorem \ref{dge} (and so Theorem \ref{e}). A natural question is that whether the constant $C$ in \eqref{113} can be independent about $K$. Notice that if we only consider the logarithmic gradient estimate of elliptic equation of \eqref{KPP}, we have proved this in \cite{LU3}. A more systemic research \cite{LU4} provides logarithmic gradient estimate and universal bounds for semilinear elliptic equations with Sobolev subcritical index on general Riemannian manifolds. Theoren \ref{dl0e} and Theorem \ref{e} can be seen as  complements for \cite{LU4}\footnote{Concretely, in \cite{LU4}, we consider the  general nonlinear item $f(u)$ satisfying $\Lambda=\sup\limits_{(0,\infty)}tf'(t)/f(t)<p_S(n)$, where $f(t)>0$ and $p_S(n)=\infty$ for $n\in[1,2]$ and $p_S(n)=\frac{n+2}{n-2}$ for $n\in(2,\infty)$. Then under some additional assumptions for nonlinear item, we establish logarithmic gradient estimate and universal bounds for such semilinear elliptic equations on Riemannian manifolds with Ricci curvature bounded below.}.\\
		(2) When we consider equation \eqref{KPP} on Euclidean spaces, more precise asymptotic estimates can be obtained if we restrict the form of solutions. Indeed, there have been classical works \cite{AW,BN,HN,HN2} that provide more precise estimates under certain assumptions on the solutions.\\
		(3) Before our current considerations, Cao-Liu-Pendleton-Ward \cite{C} derived differential Harnack inequalities for equation \eqref{KPP} with additional conditions $u\in [0,1]$ and $Ric\geq 0$. However, their inequalities can not yield the Liouville properties as presented here. On the other hand, in \cite{LU1}, we derived differential Harnack inequalities for general semilinear parabolic equations, but those estimates also need an upper bound condition and  cannot yield the Liouville property for equation \eqref{KPP}. It is worth noting that the main applications of our general framework in \cite{LU1} are sharp differential Harnack inequalities for logarithmic nonlinear terms and sublinear polynomial-type nonlinear terms, which are distinct from the Fisher-KPP type nonlinear term.
		
	\end{rmk}

	\section{Proofs of main results}
	In this section, we construct the auxiliary functions and use Bernstein-Yau method to prove results in the introduction.

	Let $f=\ln u$, then equation \eqref{KPP} becomes 
	\begin{equation}\label{feq}
		\p f=\d_V f+|\nabla f|^2+\left(a+be^{(p-1)f}\right).
	\end{equation}
	We define 
	\begin{equation}
		\L=\d_V-\p,
	\end{equation}
	and  set the auxiliary function as 
	\begin{equation}\label{F}
		F=\gamma|\nabla f|^2-\p f+\delta\left(a+be^{(p-1)f}\right),
	\end{equation}
	where $\g,\delta\in\mathbb{R}$. Then direct computation provides the following lemma.

	\begin{lem}\label{kl}
		Let $u$ be a positive solution of equation \eqref{KPP} on $B(x_0,R)\times [0,T_0]$, and we set $f=\ln u$ and $F$ is defined by \eqref{F}, then for any $m>n$, we have
		\begin{eqnarray}\label{111}
			\L F&\ge&2\g \left|\nabla^2 f-\frac{\d f}{n}g\right|^2+2\g Ric^m_V\left(\nabla f,\nabla f\right)-2\left\langle\nabla f,\nabla F\right\rangle\nonumber\\
			&&+\frac{4}{m}a\g(1-\g)(1-\delta)|\nabla f|^2+\frac{2}{m}\g(1-\delta)^2\left(a+be^{(p-1)f}\right)^2\nonumber\\
			&&-b\left(\left(\g-\delta\right)\left(p-1\right)-\delta(p-1)^2-\frac{4}{m}\g(1-\g)(1-\delta)\right)e^{(p-1)f}|\nabla f|^2\nonumber\\
			&&+\left(\frac{4}{m}\g(1-\delta)\left(a+be^{(p-1)f}\right)-b(p-1)e^{(p-1)f}\right)F\nonumber\\
			&&+\frac{2}{m}\g(1-\g)^2|\nabla f|^4+\frac{2}{m}\g F^2+\frac{4}{m}\g(1-\g)|\nabla f|^2F
		\end{eqnarray}
		on $B(x_0,R)\times[0,T_0]$.
	\end{lem}
	
	\begin{proof}
		We rewrite \eqref{feq} as 
		\begin{equation}\label{f}
			\L f=-|\nabla f|^2-a-be^{(p-1)f}.
		\end{equation}
		Using chain rule, we  see 
		\begin{equation}\label{s}
			\L F=\g\L|\nabla f|^2-(\L f)_t+\delta\L (a+be^{(p-1)f}).
		\end{equation}
		Now, we compute each item respectively.
			By Bochner's formula, we have
		\begin{eqnarray}\label{B}
			\L |\nabla f|^2&=&2|\nabla^2 f|^2+2Ric_V(\nabla f,\nabla f)+2\left\langle\nabla\L f,\nabla f\right\rangle\nonumber\\
			&=&2|\nabla^2 f|^2+2Ric_V(\nabla f,\nabla f)-2\left\langle\nabla|\nabla f|^2,\nabla f\right\rangle\nonumber\\
			&&-2b(p-1)e^{(p-1)f}|\nabla f|^2.
		\end{eqnarray}
		Chain rule directly derives
		\begin{eqnarray}\label{2g}
			\begin{cases}
				(\L f)_t= -2\left\langle\nabla f,\nabla f_t\right\rangle-b(p-1)e^{(p-1)f}f_t,\\
				\L (a+be^{(p-1)f})=b(p-1)e^{(p-1)f}\L f+b(p-1)^2e^{(p-1)f}|\nabla f|^2.
			\end{cases}
		\end{eqnarray}
		We also notice that 
		\begin{equation}\label{fF}
			\left\langle\nabla f,\nabla F\right\rangle=\g\left\langle\nabla f,\nabla |\nabla f|^2\right\rangle-\left\langle\nabla f,\nabla f_t\right\rangle+\delta b(p-1)e^{(p-1)f}|\nabla f|^2.
		\end{equation}
		Substituting \eqref{B} and \eqref{2g} into \eqref{s}, and using the equation \eqref{f}, \eqref{fF} and definition of $F$, we obtain
		\begin{eqnarray}\label{210}
			\L F&=&2\g|\nabla^2 f|^2+2\g Ric_V(\nabla f,\nabla f)-2\left\langle\nabla f,\nabla F\right\rangle\nonumber\\
			&&+\left(b(\delta-\g)(p-1)e^{(p-1)f}+b\delta(p-1)^2e^{(p-1)f}\right)|\nabla f|^2\nonumber\\
			&&-b(p-1)e^{(p-1)f}F.
		\end{eqnarray}
		Last, we observe  the following inequality ($m>n$)
		\begin{eqnarray}\label{id}
			\left|\nabla^2 f\right|^2&=&\left|\nabla^2 f-\frac{\d f}{n}g\right|^2+\frac{1}{n}(\d f)^2\nonumber\\
		&=&\left|\nabla^2 f-\frac{\d f}{n}g\right|^2+\frac{1}{n}(\d_V f-Vf)^2\nonumber\\
			&\ge&\left|\nabla^2 f-\frac{\d f}{n}g\right|^2+\frac{1}{m}(\d_V f)^2-\frac{1}{m-n}(Vf)^2
		\end{eqnarray}
		and the equivalent form of definition $F$
		\begin{equation}\label{eF}
			\d_V f=(\g-1)|\nabla f|^2-F+(\delta-1)(a+be^{(p-1)f}).
		\end{equation}
		Substitute \eqref{id} and \eqref{eF} into \eqref{210}, we will derive the desired formula.

	\end{proof}

	From Bakry-Qian's Laplacian comparison theorem \cite[Theorem 4.2]{BQ},  using the standard construction argument, we have the following  cut-off function under Bakry-\'{E}mery curvature condition.

	\begin{lem}\label{cut}
		Let $(\M,g)$ be an $n$-dimensional Riemannian manifold with $Ric_V^m\ge-Kg$, where $V$ is a smooth vector field, $m>n$ and $K\ge 0$. Then for any $R>0$, we have a cut-off function $\Phi(x)\in {\rm Lip}(B(x_0,2R))$ such that\vspace{2mm}\\
		\rm(i) $\Phi(x)=\phi(d(x_0,x))$, where  $d(x_0,\cdot)$ is the distance function from $x_0$ and  $\phi$ is a non-increasing function on $[0,\infty)$ and
		\begin{eqnarray}
			\Phi(x)=
			\begin{cases}
				1\qquad\qquad\qquad\,\text{if}\qquad x\in B(x_0,R)\nonumber\\
				0\qquad\qquad\qquad\,\text{if}\qquad x\in B(x_0,2R)\setminus B(x_0,\frac{3}{2}R).
			\end{cases}
		\end{eqnarray}
		(ii) 
		\begin{equation}
			\frac{|\nabla \Phi|}{\Phi^{\frac{1}{2}}}\le\frac{C}{R}.\nonumber
		\end{equation}
		(iii)
		\begin{equation}
			\Delta \Phi\ge-\frac{C}{R}\cdot \sqrt{mK}\coth\left(R\cdot \sqrt{\frac{K}{m}}\right)-\frac{C}{R^2}\ge-C(m)\frac{1+\sqrt{K}R}{R^2}\nonumber
		\end{equation}
		holds on $B(x_0,2R)$ in the distribution sense and pointwise outside cut locus of $x_0$. Here, $C$ is a universal constant.
	\end{lem}

	Using Lemma \ref{kl}, we provide a proof of Theorem \ref{dl}, which serves as a standard model for proving the subsequent results.

	\begin{proof}[Proof of Theorem \ref{dl}]
		If $Ric^m_V\ge-Kg$, $a>0,b\le0$ and $p>1$, then we choose $\g=\g(m,p,a,K)\in(0,1)$ and $\delta=\delta(m,a,K)<0$ such that 
		\begin{equation}\label{25}
			\frac{2a}{m}(1-\g)(1-\delta)-K\ge 0,
		\end{equation}
		\begin{equation}\label{26}
			\left(\g-\delta\right)\left(p-1\right)-\delta(p-1)^2-\frac{4}{m}\g(1-\g)(1-\delta)\ge0,
		\end{equation}
		and 
		\begin{equation}\label{27}
			L:=\frac{4}{m}\g(1-\delta)\left(a+be^{(p-1)f}\right)-b(p-1)e^{(p-1)f}\ge0.
		\end{equation}
		Therefore, by Lemma \ref{kl}, substituting \eqref{25}-\eqref{27} into \eqref{111} yields
		\begin{eqnarray}
			\L F&\ge&-2\left\langle\nabla f,\nabla F\right\rangle+LF\nonumber\\
			&&+\frac{2}{m}\g(1-\g)^2|\nabla f|^4+\frac{2}{m}\g F^2+\frac{4}{m}\g(1-\g)|\nabla f|^2F
		\end{eqnarray}
		on $B(x_0,2R)\times[0,T_0]$, where $f,F$ are defined as in Lemma \ref{kl}.
		
		Now, for any $T\in(0,T_0]$,	we define
		\begin{equation}
			G(x,t)=t\P(x)F(x,t)\qquad\text{on $\B2T$},
		\end{equation}
		where $\P(x)$ is an undetermined cut-off function as in Lemma \ref{cut}. Then
		\begin{equation}
			\L G=-\frac{1}{t}G+\frac{\d \P}{\P}G+2t\left\langle\nabla \P,\nabla F\right\rangle+t\P\L F
		\end{equation}
		on $N:=\{(x,t)\in B(x_0,2R)\times[0,T]:G(x,t)>0 \,\,\text{and}\,\,\, x\notin cut(x_0)\}$, where $cut(x_0)$ is the cut locus of $x_0$.
		
		Without loss of generality, we assume that $G(x^{\star},\tx)=\max\limits_{B(x_0,2R)\times[0,T]} G>0$ and the space component of maximum value point $\x$ is outside of cut locus of $x_0$ (see Li-Yau \cite{LY}). At this maximum value point, we see
		\begin{eqnarray}
			0&\ge&-\frac{1}{t}G+\left(\frac{\d \P}{\P}-2\frac{|\nabla\P|^2}{\P^2}\right)G+2\frac{G}{\P}\left\langle\nabla f,\nabla \P\right\rangle\nonumber\\
			&&+\frac{2}{m}\g t\P F^2+\frac{4}{m}\g(1-\g)|\nabla f|^2G.
		\end{eqnarray}
		Cauchy-Schwarz inequality and basic inequality derive
		\begin{equation}
			2\frac{G}{\P}\left\langle\nabla f,\nabla \P\right\rangle\ge-\frac{4}{m}\g(1-\g)|\nabla f|^2G-\frac{m}{4\g(1-\g)}\frac{|\nabla\P|^2}{\P^2}.
		\end{equation}
		So, we have
		\begin{equation}
			0\ge-\P G+\left(\d \P-\left(2+\frac{m}{4\g(1-\g)}\right)\frac{|\nabla\P|^2}{\P}\right)tG+\frac{2}{n}\g G^2.
		\end{equation}
		Using the property of cut-off function in Lemma \ref{cut} and $\tx\le T$, we see 
		\begin{equation}
			\max\limits_{B(x_0,2R)\times[0,T]} G=G(\x,\tx)\le C(m,\g)\left(1+\frac{1+\sqrt{K}R}{R^2}T\right).
		\end{equation}
		Therefore, we have
		\begin{equation}
			\sup\limits_{B(x_0,R)}F(\cdot,T)\le C(m,p,a,K)\left(\frac{1}{T}+\frac{1+\sqrt{K}R}{R^2}\right).
		\end{equation}
		Because  $T\in(0,T_0]$ is arbitrary, we complete the proof.

	\end{proof}

	Using Theorem \ref{dl}, we readily verify Corollary \ref{c12}.
	
	\begin{proof}[Proof of Corollary \ref{c12}]
		If the solution $u$ defines on $\M\times[0,T_0]$, letting $R\to \infty$ in \eqref{13} yields
		\begin{equation}\label{216}
			\g\frac{|\nabla u|^2}{u^2}-\frac{u_t}{u}+\delta\left(a+bu^{p-1}\right)	\le \frac{C(m,p,a,K)}{t}\quad\text{on $\M\times[0,T_0]$}.
		\end{equation}
		If $u$ is  a positive solution on $\M\times(-\infty,0)$, for any $t_0<0$ and $T>-t_0$, we see $u_T(x,t):=u(x,t-T)$ solves \eqref{KPP} on $\M\times(-\infty,T)$. So, by \eqref{216}, we have 
		\begin{eqnarray}\label{217}
			&&\left(\g\frac{|\nabla u|^2}{u^2}-\frac{u_t}{u}+\delta\left(a+bu^{p-1}\right)\right)(x,t_0)\nonumber\\
			&=&		\left(\g\frac{|\nabla u_T|^2}{u_T^2}-\frac{\p u_T}{u_T}+\delta\left(a+bu_T^{p-1}\right)\right)(x,t_0+T)\nonumber\\
			&\le& \frac{C(m,p,a,K)}{t_0+T}.\nonumber
		\end{eqnarray}
		By letting $T\to\infty$, we see
		\begin{equation}\label{218}
			\delta\left(au+bu^{p}\right)	\le \p u\quad\text{on $\M\times(-\infty,0)$}.
		\end{equation}
		For any $\e>0$, from differential inequality 
		\eqref{218}, we know that there exists a  $\tau=\tau(\e,\delta,b,p)<0$ such that $u(x,t)\le\e+\left(-\frac{a}{b}\right)^{\frac{1}{p-1}}$ when $t\le\tau$. Hence $\limsup\limits_{t\to-\infty}u(x,t)\le\e+\left(-\frac{a}{b}\right)^{\frac{1}{p-1}}$, and then letting $\e\to 0$ yields desired estimate.
		By translating time, we obtain the uniform upper bound for eternal solutions of \eqref{KPP}.
	\end{proof}

	\begin{proof}[Proof of Theorem \ref{dl0}] 
		If $Ric^m_V\ge 0$, $a=0,b\le0$ and $p>1$, then we choose $\g=\g(m,p)\in(0,1)$ and $\delta=-1<0$ such that 
		\begin{equation}\label{260}
			\left(\g-\delta\right)\left(p-1\right)-\delta(p-1)^2-\frac{4}{m}\g(1-\g)(1-\delta)\ge0,
		\end{equation}
		and 
		\begin{equation}\label{270}
			L:=\frac{4}{m}\g(1-\delta)be^{(p-1)f}-b(p-1)e^{(p-1)f}\ge0.
		\end{equation}
		Therefore, by Lemma \ref{kl}, substituting \eqref{260}-\eqref{270} into \eqref{111} yields
		\begin{eqnarray}
			\L F&\ge&-2\left\langle\nabla f,\nabla F\right\rangle+LF\nonumber\\
			&&+\frac{2}{m}\g(1-\g)^2|\nabla f|^4+\frac{2}{m}\g F^2+\frac{4}{m}\g(1-\g)|\nabla f|^2F
		\end{eqnarray}
		on $B(x_0,2R)\times[0,T_0]$.
		
		Then by nearly same arguments as proof of Theorem \ref{dl} (with $K=0$), we will finish the proof of Theorem \ref{dl0}.

	\end{proof}
	
	Notice that the condition ``$Ric^m_V\ge 0$'' is used to remove the item $|\nabla f|^2$ and is not required at elliptic case.
	
	\begin{proof}[Proof of Theorem \ref{dl0e}]
		By letting $\delta=-1$ and $\g\in(0,1)$ as in proof of Theorem \ref{dl0}, we have (notice $\p u=0$ at present case and $F=\g|\nabla f|^2-bu^{p-1}$ is nonnegative)
		\begin{eqnarray}
			\t_V F&\ge&-2\left\langle\nabla f,\nabla F\right\rangle-2KF\nonumber\\
			&&+\frac{2}{m}\g(1-\g)^2|\nabla f|^4+\frac{2}{m}\g F^2+\frac{4}{m}\g(1-\g)|\nabla f|^2F
		\end{eqnarray}
		on $B(x_0,2R)$.
		
		Same cut-off function $\Phi$ is defined as in the proof of Theorem \ref{dl} and the auxiliary function is given by $G(x)=\P(x)F(x)$ on $B(x_0,2R)$. Then nearly same arguments as proof of Theorem \ref{dl} derives desired results and we omit these abundant details.
		
	\end{proof}

	\begin{proof}[Proof of Corollary \ref{cl0}]
		Using Theorem \ref{dl0} and same arguments as proof of Corollary \ref{c12} provide the following inequality for ancient solution of \eqref{KPP} with $a=0,b<0$ and $p>1$: 
		\begin{equation}\label{222}
			\p u\ge -bu^{p} \qquad\text{on $\M\times(-\infty,0)$},
		\end{equation}
		which implies $\d u\ge-2bu^{p}\ge 0$ and so $\t u=0$ ($\M$ is closed). Therefore, $u$ must be trivial and  this is impossible.

	\end{proof}
	
	\begin{proof}[Proof of Corollary \ref{nc0}]
		Using \eqref{222} and ODE comparison theorem, one can derive the decay estimate in Corollary \ref{nc0}.

	\end{proof}
	
	\begin{proof}[Proof of Theorem \ref{dg}]
		(1) If we first choose $\delta=0$, and $\g\in\left(1-\frac{m}{4}(p-1),\frac{1}{2}\right)$ if $p\in(1+\frac{2}{m},1+\frac{4}{m})$ or $\g\in(0,1-\frac{mK}{2a})$ if $p\in[1+\frac{4}{m},\infty)$, because of the condition $K<L(m,p,a)$, then we  clearly see
		\begin{equation}\label{252}
			\frac{2a}{m}(1-\g)-K> 0,
		\end{equation}
		\begin{equation}\label{262}
			\g\left(p-1\right)-\frac{4}{m}\g(1-\g)>0,
		\end{equation}
		and 
		\begin{equation}\label{272}
			L:=\frac{4}{m}\g\left(a+be^{(p-1)f}\right)-b(p-1)e^{(p-1)f}>0,
		\end{equation}
		where $p-1-\frac{4}{m}\g>0$. 
	Therefore, there exists a  $\delta_0=\delta(m,p,a,K)>0$ sufficiently small such that when $\delta\in[-\delta_0,\delta_0]$ \eqref{252}-\eqref{272} also hold (with same $\g$).
	Now, using Lemma \ref{kl}, and  substituting \eqref{252}-\eqref{272} into \eqref{111} yields
		\begin{eqnarray}\label{226}
			\L F&\ge&-2\left\langle\nabla f,\nabla F\right\rangle+LF\nonumber\\
			&&+\frac{2}{m}\g(1-\g)^2|\nabla f|^4+\frac{2}{m}\g F^2+\frac{4}{m}\g(1-\g)|\nabla f|^2F
		\end{eqnarray}
		on $B(x_0,2R)\times[0,T_0]$. Then by same arguments as proof of Theorem \ref{dl}, we can derive the following 
		differential Harnack inequality
		\begin{eqnarray}
			\left(\g\frac{|\nabla u|^2}{u^2}-\frac{u_t}{u}\pm\delta_0\left(a+bu^{p-1}\right)\right)(x,t)	\le C\left(\frac{1}{t}+\frac{1+\sqrt{K}R}{R^2}\right)
		\end{eqnarray}
		on $B(x_0,R)\times[0,T_0]$, where $C=C(m,p,a,K)$. So we finish the proof.\\
		(2) Without loss of generality, we assume $\max(\frac{4}{m},1)<p\le1+\frac{2}{m}$ (so $m>2$) or desired inequality is obtained by (1) (notice that $K=0$ at present case). We choose
		\begin{eqnarray}
			\begin{cases}
				\g_0=1\nonumber\\
				\delta_0=1-\frac{m}{4}(p-1),\nonumber
			\end{cases}
		\end{eqnarray}
		then if $(\g,\delta)=(\g_0,\delta_0)$, we have
		\begin{eqnarray}
			\begin{cases}
				\text{LHS of \eqref{262}}=\frac{m}{4}\left(p-\frac{4}{m}\right)(p-1)^2>0\nonumber\\
				\text{LHS of \eqref{272}}=a(p-1)>0.\nonumber
			\end{cases}
		\end{eqnarray}
		Therefore, if we choose $\g_1=\g_1(m,p,a)\in(\frac{1}{2},1)$ and $\g_1$ is sufficiently close to 1, and define $\delta_1=1-\frac{m}{4\g_1}(p-1)>0$, then we also have
		\begin{eqnarray}
			\begin{cases}
				\text{LHS of \eqref{262}}>0\nonumber\\
				\text{LHS of \eqref{272}}>0.\nonumber
			\end{cases}
		\end{eqnarray}
		By following the proof of Theorem \ref{dl}, we can derive the following differential Harnack inequality
		\begin{eqnarray}
			\left(\g_1\frac{|\nabla u|^2}{u^2}-\frac{u_t}{u}+\delta_1\left(a+bu^{p-1}\right)\right)(x,t)	\le C\left(\frac{1}{t}+\frac{1+\sqrt{K}R}{R^2}\right)
		\end{eqnarray}
		on $B(x_0,R)\times[0,T_0]$, where $C=C(m,p,a)$. Now, we set $\delta_2=-\delta_1$ and choose sufficiently small and positive $\gamma_2$ such that \eqref{25}-\eqref{27} (with $K=0$) hold. So same arguments as proof of Theorem \ref{dl}, we have
		\begin{eqnarray}
			\left(\g_2\frac{|\nabla u|^2}{u^2}-\frac{u_t}{u}-\delta_1\left(a+bu^{p-1}\right)\right)(x,t)	\le C\left(\frac{1}{t}+\frac{1+\sqrt{K}R}{R^2}\right)
		\end{eqnarray}
		on $B(x_0,R)\times[0,T_0]$, where $C=C(m,p,a)$. Then we finish the proof with $\g=\min\left(\g_1,\g_2\right)$ and $\delta=\delta_1$.
	\end{proof}

	\begin{proof}[\text{Proof of Corollary \ref{pliou}}]
		If condition (a) or (b) of Corollary \ref{pliou} holds, then \eqref{19} or \eqref{110} is  valid and we both have (see the standard argument in proof of Corollary \ref{c12})
		\begin{equation}\label{231}
			\p \ln u\ge \delta|a+bu^{p-1}|\ge 0\qquad\text{on $\M\times(-\infty,0)$.}
		\end{equation}
		 By differential inequality \eqref{231}, we have $\limsup\limits_{t\to-\infty}u\le \left(-\frac{a}{b}\right)^{\frac{1}{p-1}}$. Therefore, if (1) is invalid, there exist $t_0<0$ and $l<\left(-\frac{b}{a}\right)^{\frac{1}{p-1}}$ such that $u(x,t)<l$ on $(-\infty,t_0)$ by Corollary \ref{c12} and \eqref{231}. Therefore, $\p u\ge cu$ on $\M\times(-\infty,t_0)$, where $c=c(a,b,\delta,l)>0$. And so we have the decay estimate as desired.

	\end{proof}
	

	\begin{proof}[\text{Proof of Corollary \ref{pliouf}}]
		As proof of Corollary \ref{pliou}, either (a) or (b) holds, we have
		\begin{eqnarray}\label{244}
			\left(\g\frac{|\nabla u|^2}{u^2}-\frac{u_t}{u}+\delta|a+bu^{p-1}|\right)(x,t)	\le \frac{C}{t}\qquad\text{on $\M\times(0,\infty)$}.
		\end{eqnarray}
		Here, $C$, $\g$ and $\delta$ are positive numbers which depend $m,p,a,K$ at most. First, we claim that for any $x\in\M$, $\limsup\limits_{t\to\infty}u(x,t)\le \left(-\frac{a}{b}\right)^{\frac{1}{p-1}}$. If not, there exists a $\e>0$ and a sequence $t_n\to\infty$ such that  $u(x,t_n)>\left(-\frac{a}{b}\right)^{\frac{1}{p-1}}+\e$. So $\exists n_0\in\mathbb{N}$ such that when $n\ge n_0$,
		\begin{eqnarray}\label{2450}
			\frac{u_t}{u}(x,t_n)&\ge& \delta|a+bu^{p-1}|(x,t_n)- \frac{C}{t_n}\nonumber\\
			&\ge& C(\delta,a,b,\e)- \frac{C}{t_n}\nonumber\\
			&\ge& \frac{1}{2}C(\delta,a,b,\e)>0.
		\end{eqnarray}
		Consider the Cauchy problem:
		\begin{eqnarray}\label{ca1}
			\begin{cases}
				w_t=\delta|aw+bw^p|-\frac{Cw}{t}\\
				w(t_{n_0})=u(t_{n_0}),
			\end{cases}
		\end{eqnarray}
		where $C$ is the $C(\delta,a,b,\e)$. By \eqref{2450} and \eqref{ca1}, it is clear that $w$ is increasing on its existence interval and will blow up in finite time after $t_{n_0}$. So ODE comparison theorem tells us that $u(x,\cdot)$ will also blow up, which yields a contradiction.
		
		Similarly, we also claim that for any $x\in\M$, $\liminf\limits_{t\to\infty}u(x,t)\ge \left(-\frac{a}{b}\right)^{\frac{1}{p-1}}$.  If not, there exists a $\e>0$ and a sequence $t_n\to\infty$ such that  $0<u(x,t_n)<\left(-\frac{a}{b}\right)^{\frac{1}{p-1}}-\e$. So $\exists n_0\in\mathbb{N}$ such that when $n\ge n_0$,
		\begin{eqnarray}\label{2460}
			\frac{u_t}{u}(x,t_n)&\ge& \delta|a+bu^{p-1}|(x,t_n)- \frac{C}{t_n}\nonumber\\
			&\ge& C(\delta,a,b,\e)- \frac{C}{t_n}\nonumber\\
			&\ge& \frac{1}{2}C(\delta,a,b,\e)>0.
		\end{eqnarray}
		At present case, we also consider the Cauchy problem \eqref{ca1} and readily see that it has the unique solution in $[t_{n_0},\infty)$ and $\lim\limits_{t\to\infty}w(t)=\left(-\frac{a}{b}\right)^{\frac{1}{p-1}}$. By ODE comparison, $u(x,t)\ge w(t)$ when $t\ge t_0$. Hence it contradicts with the existence of sequence $t_n$.

	\end{proof}
	
	\begin{proof}[\text{Proof of Corollary \ref{plioue}}]
		
		As proof of Corollary \ref{pliou}, either (a) or (b) holds, we have
		\begin{eqnarray}\label{2430}
			u_t\ge \delta|au+bu^{p}|\qquad\text{on $\M\times(-\infty,\infty)$}.
		\end{eqnarray}
		Here, $C$, $\g$ and $\delta$ are positive numbers which depend $m,p,a,K$ at most. By \eqref{2430}, either $u(x,\cdot)\equiv \left(-\frac{a}{b}\right)^{\frac{1}{p-1}}$ on $(t_0,\infty)$ for some $t_0>0$, or $u(x,t)\in(0,\left(-\frac{a}{b}\right)^{\frac{1}{p-1}})$ (notice that we already know $u\le \left(-\frac{a}{b}\right)^{\frac{1}{p-1}}$ in Corollary \ref{c12}). Then using ODE comparison for \eqref{2430} directly derives desired result.

	\end{proof}

	Now, we start to prove Theorem \ref{dge}. The key point of this proof is using the priori estimate in Theorem \ref{dl}.
	
	\begin{proof}[Proof of Theorem \ref{dge}]
		For elliptic case, we see $\p u=0$ and $F=\g|\nabla f|^2+\delta (a+bu^{p-1})$ (where $f=\ln u$), and so we can rewrite Lemma \ref{kl} as follows
		\begin{eqnarray}\label{112}
			\d_V F&\ge&2\g \left|\nabla^2 f-\frac{\d f}{n}g\right|^2+2\g Ric^m_V\left(\nabla f,\nabla f\right)-2\left\langle\nabla f,\nabla F\right\rangle\nonumber\\
			&&+\left(\frac{2}{m}\g(1-\delta)^2-\frac{4}{m}(1-\delta)(1-\g)\delta\right)\left(a+be^{(p-1)f}\right)^2\nonumber\\
			&&+\frac{4}{m}(1-\delta)(a+be^{(p-1)f})F-b(\g-\delta p)(p-1)e^{(p-1)f}|\nabla f|^2\nonumber\\
			&&-b(p-1)e^{(p-1)f}F+\frac{2}{m}\g(1-\g)^2|\nabla f|^4+\frac{2}{m}\g F^2+\frac{4}{m}\g(1-\g)|\nabla f|^2F\nonumber\\
			&&
		\end{eqnarray}
		on $B(x_0,2R)$.
		
		Choose $\g=\frac{2}{3}$ and $\delta=\min(\frac{1}{2},\frac{2}{3p})$, then using $Ric_V^m\ge-Kg$ yields
		\begin{eqnarray}\label{237}
			\d_V F&\ge&-2\g K|\nabla f|^2-2\left\langle\nabla f,\nabla F\right\rangle\nonumber\\
			&&+\frac{4}{m}(1-\delta)(a+be^{(p-1)f})F-b(p-1)e^{(p-1)f}F\nonumber\\
			&&+\frac{2}{m}\g(1-\g)^2|\nabla f|^4+\frac{2}{m}\g F^2+\frac{4}{m}\g(1-\g)|\nabla f|^2F
		\end{eqnarray}
		on $B(x_0,2R)$. 
		
		Now, using Theorem \ref{dl} for elliptic case, we have
		\begin{equation}\label{238}
			a+bu^{p-1}\ge- C(m,p,a,K)\frac{1+\sqrt{K}R}{R^2}
		\end{equation}
		on $B(x_0,R)$. Without loss of generality, we assume that \eqref{238} holds on $B(x_0,2R)$. And so 
		\begin{equation}\label{239}
			-2\g K|\nabla f|^2=2K\left(\delta(a+bu^{p-1})-F\right)\ge-2KF-CK\frac{1+\sqrt{K}R}{R^2}
		\end{equation}
		on $B(x_0,2R)$, where $C=C(m,p,a,K)>0$.
		
		At the following computations, the $C$ always denotes constant which depends on $m,p,a,K$. Substituting \eqref{238} and \eqref{239} to \eqref{237}, we see
		\begin{eqnarray}\label{240}
			\d_V F&\ge&-2KF-C\left(K^2+\frac{K}{R^2}\right)-2\left\langle\nabla f,\nabla F\right\rangle\nonumber\\
			&&-C\left(K+\frac{1}{R^2}\right)F+\frac{2}{m}\g F^2+\frac{4}{m}\g(1-\g)|\nabla f|^2F
		\end{eqnarray}
		on $\{x\in B(x_0,2R):F\ge 0\}$.

		Now, we define 
		\begin{equation}
			A(x)=\Phi(x)F(x)\qquad\text{on $B(x_0,2R)$},
		\end{equation}
		where $\P$ is an undetermined cut-off function as in Lemma \ref{cut}. Then
		\begin{equation}
			\d A=\frac{\d \P}{\P}A+2\left\langle\nabla \P,\nabla F\right\rangle+\P \d F
		\end{equation}
		on $\{x\in B(x_0,2R):A(x)>0\,\,\,\text{and}\,\,x\notin cut(x_0)\}$. 	Without loss of generality, we can assume that $A(\x)=\max\limits_{B(x_0,2R)} A>0$ and the maximum value point $\x$ is outside of cut locus of $x_0$. At this maximum value point , we see
		\begin{eqnarray}\label{243}
			0&\ge&\left(\frac{\d \P}{\P}-2\frac{|\nabla\P|^2}{\P^2}\right)A-2KA-C\left(K^2+\frac{K}{R^2}\right)+2\left\langle\nabla f,\nabla \P\right\rangle F\nonumber\\
			&&-C\left(K+\frac{1}{R^2}\right)A+\frac{2}{m}\g\P F^2+\frac{4}{m}\g(1-\g)|\nabla f|^2A.
		\end{eqnarray}
		Cauchy-Schwarz inequality and basic inequality derive
		\begin{equation}\label{244}
			2\left\langle\nabla f,\nabla \P\right\rangle F(\x)\ge-\frac{4}{m}\g(1-\g)|\nabla f|^2A(\x)-\frac{m}{4\g(1-\g)}\frac{|\nabla\P|^2}{\P^2}(\x).
		\end{equation}
		Substituting \eqref{243} into \eqref{244} and multiplying $\P(\x)$ at both sides, and using the property of $\P$, we have
		\begin{eqnarray}\label{245}
			\frac{2\g}{m}A^2(\x)-C\left(K+\frac{1}{R^2}\right)A(\x)-C\left(K^2+\frac{K}{R^2}\right)\ge 0,
		\end{eqnarray}
		which implies 
		\begin{equation}
			\sup\limits_{B(x_0,R)} F\le A(\x)\le C\left(K+\frac{1}{R^2}\right).
		\end{equation}
		Therefore, we complete the proof.

	\end{proof}

	

\begin{thebibliography}{99}
		\bibitem{AW}
		Aronson, Donald G. and Weinberger, Hans F., \emph{Multidimensional nonlinear diffusion arising in population genetics.} Adv. in Math. \textbf{30} (1978), no. 1, 33–76.
		\bibitem{BQ}
		Bakry, Dominique and Qian, Zhongmin, \emph{Volume comparison theorems without Jacobi fields.} Current trends in potential theory, 115–122, Theta Ser. Adv. Math., \textbf{4}, Theta, Bucharest, 2005.
		\bibitem{BN}
		Berestycki, Henri and Nirenberg, Louis, \emph{Travelling fronts in cylinders.} Ann. Inst. H. Poincaré C Anal. Non Linéaire \textbf{9} (1992), no. 5, 497–572.
		\bibitem{C}
		Cao, Xiaodong, Liu, Bowei, Pendleton, Ian and Ward, Abigail, \emph{Differential Harnack estimates for Fisher's equation.} Pacific J. Math. \textbf{290} (2017), no. 2, 273–300. 
		\bibitem{F}
		Fisher, Ronald Aylmer, The advance of advantageous genes, Ann. Eugenics 7 (1937), pp.
		335-369.
		\bibitem{HN}
		Hamel, François and Nadirashvili, Nikolaï, \emph{Travelling fronts and entire solutions of the Fisher-KPP equation in $\mathbb{R}^n$.} Arch. Ration. Mech. Anal. \textbf{157} (2001), no. 2, 91–163.
		\bibitem{HN2}
		\bysame, \emph{Entire solutions of the KPP equation.} Comm. Pure Appl. Math. \textbf{52} (1999), no. 10, 1255–1276.
		\bibitem{LY} Li, Peter and Yau, Shing-Tung, \emph{On the parabolic kernel of the Schrödinger operator}. Acta Math. \textbf{156} (1986), no. 3-4, 153–201.
		\bibitem{LU1}
		Lu, Zhihao, \emph{Differential Harnack inequalities for semilinear parabolic equations on Riemannian manifolds I: Bakry-\'{E}mery curvature bounded below.} J. Differential Equations. \textbf{377} (2023), 469-518.
		\bibitem{LU3}
		\bysame, \emph{Liouville theorems and Harnack inequalities
			for Allen-Cahn type equation.} Preprint, \href{https://arxiv.org/abs/2308.10760}{arXiv:2308.10760}.
		\bibitem{LU4}
		\bysame, \emph{Logarithmic gradient estimate and universal bounds for semilinear elliptic equations revisited.}	 Preprint, \href{https://arxiv.org/abs/2308.14026}{arXiv:2308.14026}.
		\bibitem{KPP}
		Kolmogorov, A. N., Petrovsky, I. G. and Piskunov, N. S., \emph{Etude de l’´equation de la ´
			diffusion avec croissance de la quantit´e de mati`ere et son application `a un probl`eme
			biologique.} Bull. Univ. Moscou S´er. Internat. A 1, 1–26 (1937); English transl. in:
		Dynamics of Curved Fronts, P. Pelc´e (ed.), Academic Press, 105–130 (1988) Zbl
		0018.32106
	\end{thebibliography}

\end{document}